\documentclass[11pt,a4paper,twosides]{amsart}
\usepackage{amssymb,amsmath,amsthm,mathrsfs,graphicx,xcolor}
\allowdisplaybreaks
\theoremstyle{plain}

\newcommand{\rre}{\mathbb{R}}
\newcommand{\pt}{\partial}

\numberwithin{equation}{section}

\newtheorem{thm}{Theorem}[section]

\newtheorem{lem}[thm]{Lemma}

\theoremstyle{remark}
\newtheorem{rem}{Remark}

\makeatletter
\@addtoreset{equation}{section}
\makeatother

\title[Zakharov-Kuznetsov equation in 2D] 
{Existence of wave operators for Zakharov-Kuznetsov\\
equation in two space dimensions}

\author{Jun-ichi Segata}
\address{Faculty of Mathematics, Kyushu University, 
Fukuoka, 819-0395, Japan}
\email{segata@math.kyushu-u.ac.jp}

\keywords{Zakharov-Kuznetsov equation, scattering}
\subjclass[2010]{Primary 35Q53, Secondary 35P25}

\begin{document}

\maketitle


\begin{abstract}
In this paper, we study long time behavior of solution to 
the two dimensional Zakharov-Kuznetsov equation in the 
framework of the final state problem. 
We construct a small global solution to the Zakharov-Kuznetsov 
equation which scatters to a given free solution. 
From this result, we have the existence of wave operators for 
the Zakharov-Kuznetsov equation. 
The proof is based on the space-time resonance method developed by 
Gustafson-Nakanishi-Tsai \cite{GNT2} and 
Germain-Masmoudi-Shatah \cite{GMS1,GMS2} etc.

%

\end{abstract}

%
%
\section{Introduction}

In this paper, we study long time behavior of solution to 
the Zakharov-Kuznetsov equation 
in two dimensions:
\begin{equation} 
\pt_tu+\pt_x^3u+\pt_x\pt_y^2u=\pt_x(u^2), \qquad(t,x,y)\in\rre\times\rre^2,
\label{ZK1} 
\end{equation} 
where
$u:\rre\times\rre^{2}\to\mathbb{R}$ is an unknown
function. 
Equation (\ref{ZK1}) in the three dimensional case was 
derived by Zakharov-Kuznetsov \cite{ZK} to describe 
unidirectional wave propagation in a magnetized plasma. 
Equation (\ref{ZK1}) in the two dimensional case was 
derived by Laedke-Spatschek \cite{LSp} from the basic 
hydrodynamical equations. 
Lannes-Linares-Saut \cite{LLS} gave the rigorous justification 
of (\ref{ZK1}) from the Euler-Poisson system for a uniformly 
magnetized media. 

Equation (\ref{ZK1}) has the conservation of mass
\begin{equation} \label{mass}
\|u(t)\|_{L_{x}^{2}} = \|u(0)\|_{L_{x}^{2}},
\end{equation}
and the conservation of energy
\begin{eqnarray} 
E[u](t)&:=&\frac12
\int\!\!\!\int_{\rre^2}\left\{(\pt_xu)^2+(\pt_yu)^2\right\}dxdy
-\frac16\int\!\!\!\int_{\rre^2}u^3 dxdy=E[u](0).
\label{energy}
\end{eqnarray}

The main purpose of this paper is to study the asymptotic behavior 
in time of solution to (\ref{ZK1}). 
We first review the related results on the well-posedness and 
scattering. 
For the three dimensional case, Linares-Saut \cite{LS} 
proved the time local well-posedness in $H^s(\rre^3)$ 
with $s>9/8$. 
Ribaud-Vento \cite{RV} 
has shown the local well-posedness in $H^s(\rre^3)$ 
with $s>1$ by using the local smoothing effect and the maximal 
function estimates for the solution to the linearized equation of (\ref{ZK1}). 
Recently, Herr-Kinoshita \cite{HK2} proved the local well-posedness 
in $H^s(\rre^3)$ with full scaling sub-critical range 
$s>-1/2$ by the Fourier restriction norm method 
(see also \cite{HK2} for the well-posedness results of (\ref{ZK1}) 
for the high dimensions). 
For the two dimensional case, Faminskii \cite{F1} 
proved the global well-posedness in $H^1(\rre^2)$ by 
applying the contraction mapping principle for the corresponding 
integral equation with an aid of the local smoothing effect 
for the linearized equation. 
Gr\"unrock-Herr \cite{GH} and Molinet-Pilod \cite{MP} 
independently extended this result to $s>1/2$ 
by using the Fourier restriction norm method. 
Furthermore, Kinoshita \cite{K} proved the local 
well-posedness in $H^s(\rre^2)$ with $s>-1/4$ 
by refining the Fourier restriction norm method. 
Especially, from his result and the conservation of mass 
(\ref{mass}), we have the global well-posedness of 
(\ref{ZK1}) in $L^2(\rre^2)$.

To summarize the related results on the scattering, we 
consider the generalized Zakharov-Kuznetsov equation
\begin{equation} 
\pt_tu+\pt_{x_1}\Delta u=\pt_{x_1}(u^k), \qquad(t,x)\in\rre\times\rre^d,
\label{gZK} 
\end{equation} 
where $x=(x_1,\cdots,x_d)\in\rre^d$, $d$, $k\ge2$ are integers, 
and $\Delta$ is the Laplacian with respect to $x$. 
For the high dimensional case, Herr-Kinoshita \cite{HK1} 
proved the small data scattering 
for (\ref{gZK}) with $d\ge5$ and $k=2$ in the scaling critical Sobolev space. 
Furthermore, they proved the  scattering for (\ref{gZK}) with $d=4$ and $k=2$ 
when the initial data is small and radial in the last $(d-1)$ variables. 
For the two dimensional case, Farah-Linares-Pastor \cite{FLP} 
have shown the small data scattering for (\ref{gZK}) with $d=2$ 
and $k\ge4$. Very recently, Anjolras \cite{A} proved 
the small data scattering for the two dimensional 
modified Zakharov-Kuznetsov equation 
(i.e., (\ref{gZK}) with $d=2$ and $k=3$) 
by using the space-time resonance method developed by 
Gustafson-Nakanishi-Tsai \cite{GNT2} and 
Germain-Masmoudi-Shatah \cite{GMS1,GMS2} etc. 
Furthermore, by using the Fourier restriction norm method, 
Kinoshita-Correia\footnote{work in preparation.} obtained 
the small data scattering for the two dimensional 
modified Zakharov-Kuznetsov equation under the situation that 
the initial and the asymptotic states belong to the same class. 

To study the scattering problem for (\ref{ZK1}) in two dimensions 
is more difficult compared to (\ref{gZK}) with $k>1+2/d$ due 
to slow decay of the solution to the linearized equation 
and low power of the nonlinearity. Furthermore, 
from the point of view of the linear scattering theory, 
for the two dimensional case, the quadratic 
nonlinearity in (\ref{ZK1}) belongs to borderline between 
the short and long range cases. Therefore it is non-trivial 
to obtain long time behavior of solution to (\ref{ZK1}) 
in two space dimensions. 
In this paper, we consider the scattering problem for 
(\ref{ZK1}) in the framework of the final state problem. 
More precisely, we prove the existence of small global 
solutions to (\ref{ZK1}) which scatters to a given free solution. 

Following \cite{BKS,GH}, to study (\ref{ZK1})  
we first reduce (\ref{ZK1}) into the equation which is 
symmetric with respect to $x$ and $y$. Let 
$x':=2^{-2/3} x+2^{-2/3}\cdot 3^{1/2} y$, $y':=2^{-2/3} x-2^{-2/3}\cdot 3^{1/2} y$ 
and let $v(x',y'):=2^{-2/3}u(x,y)$. 
By a simple calculation, we see that 
if $u$ is a solution to (\ref{ZK1}), then 
$v$ satisfies 
\begin{eqnarray*}
\pt_tv+(\pt_{x'}^3+\pt_{y'}^3)v
=
(\pt_{x'}+\pt_{y'})(v^2).
\end{eqnarray*}
Henceforth, we consider the scattering 
problem for 
\begin{eqnarray}
\pt_tv+(\pt_{x_1}^3+\pt_{x_2}^3)v
=
(\pt_{x_1}+\pt_{x_2})(v^2),
\qquad(t,x_1,x_2)\in\rre\times\rre^2.
\label{ZK}
\end{eqnarray}
To state the our main theorem, we introduce 
several notation. We define the norm $\|\cdot\|_X$ 
by
\begin{eqnarray}
\|f\|_{X}
&:=&
\|\langle x\rangle \pt_{x_1}^{-\frac12}\pt_{x_2}^{-\frac12}f\|_{W_x^{3,1}}
+\|\langle x\rangle \pt_{x_1}^{-1}f\|_{H_x^3}
+\|\langle x\rangle \pt_{x_2}^{-1}f\|_{H_x^3}
\label{normX}\\
& &+\|\langle x\rangle \pt_{x_1}\pt_{x_2}^{-2}f\|_{H_x^3}
+\|\langle x\rangle \pt_{x_1}^{-2}\pt_{x_2}f\|_{H_x^3}
+\|\pt_{x_1}^{-2}f\|_{H_x^3}
+\|\pt_{x_2}^{-2}f\|_{H_x^3}, 
\nonumber
\end{eqnarray}
where $x=(x_1,x_2)$ and $\langle x\rangle=\sqrt{1+|x|^2}$. 
Let $V(t):=e^{-t(\pt_{x_1}^3+\pt_{x_2}^3)}$  
be a unitary group generated by 
$-\pt_{x_1}^3-\pt_{x_2}^3$. 
Then we have the following. 

\begin{thm}\label{main}
There exists $\varepsilon>0$ such that if 
$v_{\infty}$ satisfies $\|v_{\infty}\|_{X}\le\varepsilon$, 
then there exists a unique global solution 
$v\in C(\rre; H^1(\rre^2))$ to (\ref{ZK}) satisfying
\begin{eqnarray}
\|v(t)-V(t)v_{\infty}\|_{H_x^2}
\lesssim\varepsilon t^{-\alpha}\label{scattering}
\end{eqnarray}
for any $t>0$, where $\alpha>2/3$. 
Similar result holds for $t<0$.
\end{thm}

\vskip2mm

From Theorem \ref{main}, we have that for sufficiently small 
$\varepsilon>0$
the wave operator $B_{\varepsilon}\ni v_{\infty}\mapsto v(0)\in H_x^1$ 
is well-defined, where $B_{\varepsilon}=\{f\in X\ ;\ \|f\|_X\le\varepsilon\}$. 

\vskip2mm

We now give the outline of the proof of 
Theorem \ref{main}. 
For given final data $v_{\infty}$, 
we introduce a new unknown function
\begin{eqnarray*} 
w(t,x):=v(t,x)-v_1(t,x)-v_2(t,x),
\end{eqnarray*} 
where $v_1$ and $v_2$ are given by 
\begin{eqnarray} 
v_1(t,x)&=&[V(t)v_{\infty}](x),\label{v1}\\
v_2(t,x)&=&-
(\pt_{x_1}+\pt_{x_2})\int_t^{\infty}
V(t-\tau)[v_1(\tau)^2]d\tau\label{v2}\\
&=&-
(\pt_{x_1}+\pt_{x_2})\int_t^{\infty}
V(t-\tau)[(V(\tau)v_{\infty})^2]d\tau.
\nonumber
\end{eqnarray} 
Let us derive the evolution equation for $w$. 
Let ${{\mathcal L}}=\pt_t+\pt_{x_1}^3+\pt_{x_2}^3$. 
Since ${{\mathcal L}}v_1=0$ and 
${{\mathcal L}}v_2=
(\pt_{x_1}+\pt_{x_2})(v_1^2)$, 
we have   
\begin{eqnarray*} 
{{\mathcal L}}w&=&{{\mathcal L}}v-{{\mathcal L}}v_1-{{\mathcal L}}v_2\\
&=&{{\mathcal L}}v-
(\pt_{x_1}+\pt_{x_2})(v_1^2).
\end{eqnarray*} 
If $v$ satisfies (\ref{ZK}), then 
\begin{eqnarray*} 
{{\mathcal L}}v
&=&
(\pt_{x_1}+\pt_{x_2})(v^2)\\
&=&
(\pt_{x_1}+\pt_{x_2})\{(w+v_1+v_2)^2\}\\
&=&
(\pt_{x_1}+\pt_{x_2})\{w^2+2(v_1+v_2)w+v_1^2+2v_1v_2+v_2^2\}\\
&=&
(\pt_{x_1}+\pt_{x_2})(w^2)+2
(\pt_{x_1}+\pt_{x_2})\{(v_1+v_2)w\}\\
& &
+
(\pt_{x_1}+\pt_{x_2})(v_1^2+2v_1v_2+v_2^2).
\end{eqnarray*} 
Hence we see that $w$ satisfies 
\begin{eqnarray} 
\partial_tw+(\pt_{x_1}^3+\pt_{x_2}^3)w
&=&
(\pt_{x_1}+\pt_{x_2})(w^2)
+2
(\pt_{x_1}+\pt_{x_2})\left\{(v_1+v_2)w\right\}
\label{ZK11}\\
& &+
(\pt_{x_1}+\pt_{x_2})(2v_1v_2+v_2^2).
\nonumber
\end{eqnarray}
To show Theorem \ref{main}, we prove the existence  
of solution $w$ to (\ref{ZK11}) with $w\to0$ in $H^2$ 
as $t\to\infty$. We first note that 
$w$ and $v_1$ can be easily estimated by 
the energy and Strichartz estimates 
(see Lemma \ref{lemL}, below). 
The main difficulty to prove Theorem \ref{main} 
lies on estimate for $v_2$. More precisely, we need 
$L^2$ estimate for the bilinear oscillatory integral 
\begin{eqnarray} 
v_2=-
(\pt_{x_1}+\pt_{x_2})\int_t^{\infty}V(t-\tau)
\left[(V(\tau)v_{\infty})^2\right]d\tau.
\label{u1}
\end{eqnarray}  
Since the solution to the linearized equation of (\ref{ZK}) 
decays like $O(t^{-1})$ in $L^{\infty}$, the H\"{o}lder 
inequality only yields that integrand of (\ref{u1}) is  
$O(t^{-1})$ in $L^2$ which is insufficient because this is not 
integrable in time variable. To derive $L^2$ 
estimates for (\ref{u1}), we employ so called 
space-time resonance method which is developed by 
\cite{GMS1,GMS2,GNT2} etc. 
More precisely, to evaluate (\ref{u1}), we crucially use 
the ``null structure" of the nonlinear term 
which can be represented as the algebraic identity 
(\ref{key1}) below, and  
integration by parts both in space and time variables. 
In this step, we use the boundedness theorem by Coifman-Meyer 
\cite{CoifmanMeyer} for a bilinear multiplier operator. 
%
Once we obtain $L^2$ estimate for $v_2$, 
we have an existence of solution to (\ref{ZK11}) 
by the compactness argument.

Note that the space-time resonance method 
is now widely used to study usual/modified scattering 
for various nonlinear hyperbolic and dispersive PDEs in the framework of 
the initial value problem, see \cite{A,GMS1,GMS2,GNT2} 
for instance. In this paper we use this approach 
to study the scattering problem of (\ref{ZK1}) 
in the framework of the final state problem. 
We believe that for the small initial data, the 
corresponding solution to (\ref{ZK1}) scatters 
to the free solution. Proving this fact is an interesting open question.

%
%
%

We introduce several notations and function spaces 
which are used throughout this paper. 
For $f\in{{\mathcal S}}'(\rre^{d})$, $\hat{f}(\xi)$ 
denotes the Fourier transform of $f$. Let 
$\langle\xi\rangle=\sqrt{|\xi|^2+1}$. The differential operator 
$\langle\nabla\rangle^s=(1-\Delta)^{s/2}$ denotes 
the Bessel potential of order $-s$. 
For $1\le q,r\le\infty$, $L^q(t,\infty;L_x^r)$ is defined 
as follows:
\begin{eqnarray*}
L^q(t,\infty;L_x^r)&=&\{u\in{{\mathcal S}}'(\rre^{1+2});
\|u\|_{L^q(t,\infty;L_x^r)}<\infty\},\\
\|u\|_{L^q(t,\infty;L_x^r)}&=&
\left\|\|u(\tau)\|_{L_x^r}\right\|_{L_{\tau}^q(t,\infty)}.
\end{eqnarray*}
We will use the inhomogeneous Sobolev spaces
\begin{eqnarray*}
W^{s,p}=\lbrace f \in \mathcal S'(\rre^2); 
\|f\|_{W^{s,p}}=\| 
\langle \nabla\rangle^{s} f\|_{L^p}<\infty\rbrace,
\end{eqnarray*}
where $s\in\rre$ and $1\le p\le\infty$. 
We denote 
$H^s=W^{s,2}$. 
We denote $A \lesssim B$ if there exists a constant $C > 0$ 
such that $A \le C B$ holds and $ A \sim B $ if $ A \lesssim B \lesssim A $.

The outline of the paper is as follows. In Section 2, we 
give the decay and Strichartz estimates for the  
linearized equation of (\ref{ZK}). In Section 3, we derive 
the key bilinear dispersive estimates. Finally,
in Section 4, we prove Theorem \ref{main}.

\vskip2mm

%
%
\section{Linear Dispersive Estimates}

In this section we derive the linear estimates associated with 
(\ref{ZK}): 
\begin{equation} 
\left\{
\begin{array}{l}
\partial_tw+(\pt_{x_1}^3+\pt_{x_2}^3)w=0,\qquad\quad (t,x_1,x_2)\in
\mathbb{R}\times\mathbb{R}^{2}, \\
w(0,x_1,x_2)=f(x_1,x_2),\qquad\ \ \ (x_1,x_2)\in\mathbb{R}^{2}.
\end{array}
\right.\label{LZK} 
\end{equation}
Let us recall that $V(t)=e^{-t(\pt_{x_1}^3+\pt_{x_2}^3)}$  
is the  unitary group generated by 
$-\pt_{x_1}^3-\pt_{x_2}^3$. Then, the solution to (\ref{LZK}) 
can be written as $V(t)f$. 

We have the following decay and the Strichartz estimates for 
(\ref{LZK}).

\vskip2mm

\begin{lem}\label{lemL} 
(i) Let $2\le p\le\infty$. Then for any $t>0$, we have
\begin{eqnarray}
\|V(t)f\|_{L_{x}^p}
&\lesssim&t^{-\frac23(1-\frac2p)}
\|f\|_{L_{x}^{p'}},\label{linear1}\\
\||\pt_{x_1}|^{\frac12-\frac1p}
|\pt_{x_2}|^{\frac12-\frac1p}V(t)f\|_{L_{x}^p}
&\lesssim&t^{-(1-\frac2p)}
\|f\|_{L_{x}^{p'}},\label{linear2}
\end{eqnarray}
where $p'$ is the H\"older conjugate exponent of $p$.

\vskip2mm
\noindent
(ii) Let $(q_{j},r_{j})$ ($j=1,2$) satisfy 
$3/q_{j}+2/r_{j}=1$ and $2\le r_{j}\le\infty$. Then, we have 
\begin{equation}
\left\|
\int_{\tau}^{\infty}V(\tau-\tau')F(\tau')d\tau'
\right\|_{L_{\tau}^{q_{1}}(t,\infty;L_{x}^{r_{1}})}
\lesssim\|
F\|_{L_{\tau}^{q_{2}'}(t,\infty;L_{x}^{r_{2}'})}.
\label{linear3}
\end{equation}
\end{lem}

\vskip2mm

\begin{proof}[Proof of Lemma \ref{lemL}.] 
(i) By the unitary property of $V(t)$, we have 
\begin{equation}
\|V(t)f\|_{L_{x}^2}=\|f\|_{L_{x}^{2}}. \label{l2}
\end{equation}
On the other hand, $V(t)f$ can be rewritten as 
\begin{eqnarray}
[V(t)f](x_1,x_2)
&=&
\frac{1}{2\pi}
\int\!\!\!\int_{\rre^2}
e^{ix_1\xi_1+ix_2\xi_2+it(\xi_1^3+\xi_2^3)}
\widehat{f}(\xi_1,\xi_2)d\xi_1 d\xi_2
\label{v}\\
&=&
\int\!\!\!\int_{\rre^2}
K(t,x_1-x_1',x_2-x_2')f(x_1',x_2')dx_1'dx_2',
\nonumber
\end{eqnarray}
where
\begin{eqnarray*}
K(t,x_1,x_2)=
\frac{1}{(2\pi)^2}\int\!\!\!\int_{\rre^2}
e^{ix_1\xi_1+ix_2\xi_2+it(\xi_1^3+\xi_2^3)}d\xi_1 d\xi_2.
\end{eqnarray*}
Note 
\begin{eqnarray*}
K(t,x_1,x_2)&=&
\frac{1}{(2\pi)^2}
\left(\int_{\rre}e^{ix_1\xi_1+it\xi_1^3}d\xi_1\right)
\left(\int_{\rre}e^{ix_2\xi_2+it\xi_2^3}d\xi_2\right)\\
&=&
\frac{1}{2\pi}A(t,x_1)A(t,x_2),
\end{eqnarray*}
where the function $A$ is given by 
\begin{eqnarray*}
A(t,y)=\frac{1}{\sqrt{2\pi}}\int_{\rre}e^{iy\xi+it\xi^3}d\xi
=t^{-\frac13}Ai(t^{-\frac13}y),\qquad t, y\in\rre,
\end{eqnarray*}
and $Ai$ is the Airy function :  
$Ai(z)=(2\pi)^{-1/2}\int_{\rre}e^{iz\xi+i\xi^3}d\xi$. 
By the well-known property of the Airy function $Ai\in L^{\infty}$
(see \cite[Section 1.5]{LP} for instance), we have
\begin{eqnarray*}
|A(t,y)|\lesssim t^{-\frac13}.
\end{eqnarray*}
Hence 
\begin{eqnarray*}
|K(t,x_1,x_2)|\lesssim t^{-\frac23}.
\end{eqnarray*}
Substituting this into (\ref{v}), we have
\begin{eqnarray}
\|V(t)f\|_{L_{x}^{\infty}}
&\lesssim&t^{-\frac23}
\|f\|_{L_{x}^{1}}.\label{l1}
\end{eqnarray}
By the interpolation between (\ref{l2}) and (\ref{l1}), 
we have (\ref{linear1}). 

Furthermore, by the well-known property of the Airy function 
(see \cite[Section 1.5]{LP} for instance)
\begin{eqnarray*}
||\pt_y|^{\frac12}A(t,y)|\lesssim t^{-\frac12},
\end{eqnarray*}
we have
\begin{eqnarray*}
||\pt_{x_1}|^{\frac12}
|\pt_{x_2}|^{\frac12}K(t,x_1,x_2)|\lesssim t^{-1}.
\end{eqnarray*}
Substituting this into (\ref{v}), we have
\begin{eqnarray}
\||\pt_{x_1}|^{\frac12}|\pt_{x_2}|^{\frac12}V(t)f\|_{L_{x}^{\infty}}
&\lesssim&t^{-1}
\|f\|_{L_{x}^1}.\label{l3}
\end{eqnarray}
By the interpolation between (\ref{l2}) and (\ref{l3}), 
we have (\ref{linear2}). 

\vskip2mm
\noindent
(ii) The inequality (\ref{linear3}) follows 
by applying \cite[Theorem 1.2]{KT}
to (\ref{linear1}).
\end{proof}

\vskip2mm

%
%
\section{Bilinear Dispersive Estimates}

In this section we derive $L^2$ estimate for $v_2$ 
defined by (\ref{v2}) which is key to prove 
Theorem \ref{main}. We show the following.

\begin{lem}\label{xF}
For any $t>0$, we have 
\begin{eqnarray}
\left\|(\pt_{x_1}+\pt_{x_2})\int_t^{\infty}V(t-\tau)
\left[(V(\tau)f)(V(\tau)g)\right]d\tau\right\|_{L_x^2}
&\lesssim& t^{-1}\|f\|_{Z}\|g\|_{Z},
\label{xFe}
\end{eqnarray}
where 
\begin{eqnarray*}
\|f\|_{Z}&=&
\|\langle x\rangle \pt_{x_1}^{-\frac12}\pt_{x_2}^{-\frac12}f\|_{L_x^1}
+\|\langle x\rangle \pt_{x_1}^{-1}f\|_{L_x^2}
+\|\langle x\rangle \pt_{x_2}^{-1}f\|_{L_x^2}\\
& &+\|\langle x\rangle \pt_{x_1}\pt_{x_2}^{-2}f\|_{L_x^2}
+\|\langle x\rangle \pt_{x_1}^{-2}\pt_{x_2}f\|_{L_x^2}
+\|\pt_{x_1}^{-2}f\|_{L_x^2}
+\|\pt_{x_2}^{-2}f\|_{L_x^2}.
\end{eqnarray*}
\end{lem}

To show Lemma \ref{xF}, we employ the 
Coifman-Meyer theorem \cite{CoifmanMeyer} for the multilinear 
multiplier operators. The following is a bilinear version 
of the Coifman-Meyer theorem.

\begin{lem}[]\label{CM} 
Let $m(\zeta,\sigma)$ be a $C^{\infty}$ function on 
$(\rre^d\times\rre^d)\backslash\{(0,0)\}$, 
and  $T_m$ be a bilinear multiplier operator defined by
\begin{eqnarray*}
T_m[f,g](x)
=(2\pi)^{-2d}
\int\!\!\!\int_{\rre^d\times\rre^d}
e^{ix(\zeta+\sigma)}
m(\zeta,\sigma)\widehat{f}(\zeta)\widehat{g}(\sigma)
d\zeta d\sigma.
\end{eqnarray*}
Assume that $m$ satisfies the H\"{o}rmander-Mihlin condition: 
for any multi-indices $\beta\in{{\mathbb Z}}_{+}^{2d}$,
\begin{eqnarray}
|\pt_{\zeta,\sigma}^{\beta}m(\zeta,\sigma)|
\le C_{\beta}(|\zeta|^2+|\sigma|^2)^{-\frac{|\beta|}{2}}.
\label{MH}
\end{eqnarray}
Then for any $1\le p,q\le\infty$ and $1<r<1$ 
satisfying $1/r=1/p+1/q$, we have
\begin{eqnarray*}
\left\|T_m[f,g]\right\|_{L_x^r}
\lesssim\|f\|_{L_x^p}\|g\|_{L_x^q}.
\end{eqnarray*}
\end{lem}

\vskip2mm
\begin{rem} By the Plancherel theorem and change of variables 
$\zeta=\xi-\eta$, $\sigma=\eta$, 
\begin{eqnarray*}
\lefteqn{\left\|
\int_{\rre^d}m(\xi,\eta)\widehat{f}(\xi-\eta)\widehat{g}(\eta)
d\eta
\right\|_{L_{\xi}^2}}\\
&=&
\left\|
{{\mathcal F}}^{-1}_{\xi\mapsto x}
\left[\int_{\rre^d}m(\xi,\eta)\widehat{f}(\xi-\eta)\widehat{g}(\eta)
d\eta\right]\right\|_{L_x^2}\\
&=&
\left\|\int\!\!\!\int_{\rre^d\times\rre^d}
e^{ix\xi}m(\xi,\eta)\widehat{f}(\xi-\eta)\widehat{g}(\eta)
d\xi d\eta\right\|_{L_x^2}\\
&=&
\left\|\int\!\!\!\int_{\rre^d\times\rre^d}
e^{ix(\zeta+\sigma)}m(\zeta+\sigma,\sigma)
\widehat{f}(\zeta)\widehat{g}(\sigma)
d\zeta d\sigma\right\|_{L_x^2}.
\end{eqnarray*}
Hence if $\widetilde{m}(\zeta,\sigma)
:=m(\zeta+\sigma,\sigma)$ 
satisfies the H\"{o}rmander-Mihlin condition (\ref{MH}), 
then we have
\begin{eqnarray}
\left\|
\int_{\rre^d}m(\xi,\eta)\widehat{f}(\xi-\eta)\widehat{g}(\eta)
d\eta
\right\|_{L_{\xi}^2}
\lesssim\|f\|_{L_x^p}\|g\|_{L_x^q}, 
\label{rem1}
\end{eqnarray}
where $1\le p,q\le\infty$ and $1/p+1/q=1/2$. 
In the proof of Lemma \ref{xF}, we frequently use 
the inequality (\ref{rem1}).
\end{rem}

\vskip2mm

\begin{proof}[Proof of Lemma \ref{CM}.] 
See \cite{CoifmanMeyer}, \cite[Theorem 7.5.3]{G} for instance.
\end{proof}

\begin{proof}[Proof of Lemma \ref{xF}.] 
Simple calculation yields
\begin{eqnarray} 
\lefteqn{(\pt_{x_1}+\pt_{x_2})\int_t^{\infty}V(t-\tau)
\left[(V(\tau)f)(V(\tau)g)\right]d\tau}\label{o1}\\
&=&
{{\mathcal F}}_{\xi\mapsto x}^{-1}
\left[(\xi_1+\xi_2)e^{it(\xi_1^3+\xi_2^3)}
\int_t^{\infty}\!\!\!\int_{\rre^2}e^{-i\tau\phi(\xi,\eta)}
\widehat{f}(\xi-\eta)\widehat{g}(\eta)
d\eta d\tau\right](x),\nonumber
\end{eqnarray} 
where the resonant function $\phi$ is given by 
\begin{eqnarray} 
\phi(\xi,\eta)
&=&(\xi_1^3+\xi_2^3)-\left\{(\xi_1-\eta_1)^3+(\xi_2-\eta_2)^3\right\}
-(\eta_1^3+\eta_2^3)\label{r1}\\
&=&3(\xi_1^2\eta_1-\xi_1\eta_1^2)+3(\xi_2^2\eta_2-\xi_2\eta_2^2)
\nonumber\\
&=&3\xi_1(\xi_1-\eta_1)\eta_1+3\xi_2(\xi_2-\eta_2)\eta_2.
\nonumber
\end{eqnarray} 
To show (\ref{xFe}), we estimate 
\begin{eqnarray} 
I(f,g):=(\xi_1+\xi_2)
\int_t^{\infty}\!\!\!\int_{\rre^2}e^{-i\tau\phi(\xi,\eta)}
\widehat{f}(\xi-\eta)\widehat{g}(\eta)
d\eta d\tau\label{Ifg}
\end{eqnarray} 
by using the space-time resonance method.

By (\ref{r1}), we see 
\begin{eqnarray} 
\pt_{\eta_1}\phi(\xi,\eta)&=&3\xi_1^2-6\xi_1\eta_1=3\xi_1(\xi_1-2\eta_1),
\label{r2}\\
\pt_{\eta_2}\phi(\xi,\eta)&=&3\xi_2^2-6\xi_2\eta_2=3\xi_2(\xi_2-2\eta_2).
\label{r3}
\end{eqnarray} 
Especially, we have
\begin{eqnarray} 
\phi(\xi,\eta)-\eta\cdot\nabla_{\eta}\phi(\xi,\eta)
=3\xi_1\eta_1^2+3\xi_2\eta_2^2.
\label{r4}
\end{eqnarray} 
%
By (\ref{r1}), (\ref{r2}) and (\ref{r3}), we have 
\begin{eqnarray}
\lefteqn{(\xi_1^2-\xi_1\xi_2+\xi_2^2)(\xi_1+\xi_2)}\qquad\qquad\label{k1}\\
&=&\xi_1^3+\xi_2^3\nonumber\\
&=&\frac43\phi(\xi,\eta)+\frac13(\xi_1-2\eta_1)\pt_{\eta_1}\phi(\xi,\eta)
+\frac13(\xi_2-2\eta_2)\pt_{\eta_2}\phi(\xi,\eta).
\nonumber
\end{eqnarray}
Furthermore, by (\ref{r4}), we see 
\begin{eqnarray*}
\xi_2=-\xi_1\frac{\eta_1^2}{\eta_2^2}
+\frac{\phi(\xi,\eta)-\eta\cdot\nabla_{\eta}\phi(\xi,\eta)}{3\eta_2^2}.
\end{eqnarray*}
Hence
\begin{eqnarray}
\xi_1+\xi_2=\xi_1\frac{\eta_2^2-\eta_1^2}{\eta_2^2}
+\frac{\phi(\xi,\eta)-\eta\cdot\nabla_{\eta}\phi(\xi,\eta)}{3\eta_2^2}.
\label{k2}
\end{eqnarray}
Multiplying (\ref{k2}) by $(\xi_1-2\eta_1)^2$ and using (\ref{r2}), we have
\begin{eqnarray}
\lefteqn{(\xi_1-2\eta_1)^2(\xi_1+\xi_2)}\label{k3}\\
&=&
\frac13(\xi_1-2\eta_1)\frac{\eta_2^2-\eta_1^2}{\eta_2^2}
\pt_{\eta_1}\phi(\xi,\eta)
+\frac13(\xi_1-2\eta_1)^2\frac{1}{\eta_2^2}
\left(\phi(\xi,\eta)-\eta\cdot\nabla_{\eta}\phi(\xi,\eta)\right).
\nonumber
\end{eqnarray}
In a similar way, we obtain
\begin{eqnarray}
\lefteqn{(\xi_2-2\eta_2)^2(\xi_1+\xi_2)}\label{k4}\\
&=&
\frac13(\xi_2-2\eta_2)\frac{\eta_1^2-\eta_2^2}{\eta_1^2}
\pt_{\eta_2}\phi(\xi,\eta)
+\frac13(\xi_2-2\eta_2)^2\frac{1}{\eta_1^2}
\left(\phi(\xi,\eta)-\eta\cdot\nabla_{\eta}\phi(\xi,\eta)\right).
\nonumber
\end{eqnarray}
(\ref{k1})$+$(\ref{k3})$+$(\ref{k4}) yields
\begin{eqnarray}
\lefteqn{\left\{(\xi_1^2-\xi_1\xi_2+\xi_2^2)
+(\xi_1-2\eta_1)^2+(\xi_2-2\eta_2)^2\right\}
(\xi_1+\xi_2)}\label{k}\\
&=&A(\xi,\eta)\phi(\xi,\eta)
+B_1(\xi,\eta)\pt_{\eta_1}\phi(\xi,\eta)
+B_2(\xi,\eta)\pt_{\eta_2}\phi(\xi,\eta),
\nonumber
\end{eqnarray}
where 
\begin{eqnarray}
A(\xi,\eta)&=&\frac43
+\frac13(\xi_1-2\eta_1)^2\frac{1}{\eta_2^2}
+\frac13(\xi_2-2\eta_2)^2\frac{1}{\eta_1^2}
\label{termA}\\
&=:&\sum_{k=1}^3A_{k}(\xi,\eta),
\nonumber\\
B_1(\xi,\eta)&=&
\frac13(\xi_1-2\eta_1)\frac{2\eta_2^2-\eta_1^2}{\eta_2^2}
-\frac13(\xi_1-2\eta_1)^2\frac{\eta_1}{\eta_2^2}
-\frac13(\xi_2-2\eta_2)^2\frac{1}{\eta_1}
\label{termB1}\\
&=:&\sum_{k=1}^3B_{1,k}(\xi,\eta),
\nonumber\\
B_2(\xi,\eta)&=&
\frac13(\xi_2-2\eta_2)\frac{2\eta_1^2-\eta_2^2}{\eta_1^2}
-\frac13(\xi_1-2\eta_1)^2\frac{1}{\eta_2}
-\frac13(\xi_2-2\eta_2)^2\frac{\eta_2}{\eta_1^2}
\label{termB2}\\
&=:&\sum_{k=1}^3B_{2,k}(\xi,\eta).
\nonumber
\end{eqnarray}
Hence we have the key identity 
\begin{eqnarray}
\xi_1+\xi_2
=\frac{1}{p(\xi,\eta)}
\left(A(\xi,\eta)\phi(\xi,\eta)
+B_1(\xi,\eta)\pt_{\eta_1}\phi(\xi,\eta)
+B_2(\xi,\eta)\pt_{\eta_2}\phi(\xi,\eta)
\right),
\label{key1}
\end{eqnarray}
where
\begin{eqnarray*}
p(\xi,\eta)=(\xi_1^2-\xi_1\xi_2+\xi_2^2)
+(\xi_1-2\eta_1)^2+(\xi_2-2\eta_2)^2.
\end{eqnarray*}

By using (\ref{key1}), 
we split $I(f,g)$ (defined by (\ref{Ifg})) to 
the time resonant and space resonant terms:
\begin{eqnarray}
\lefteqn{I(f,g)}\label{ST}\\
&=&
\int_t^{\infty}\!\!\!\int_{\rre^2}
\phi(\xi,\eta)e^{-i\tau\phi(\xi,\eta)}\frac{1}{p(\xi,\eta)}
A(\xi,\eta)\widehat{f}(\xi-\eta)
\widehat{g}(\eta)
d\eta d\tau\nonumber\\
& &+\sum_{j=1}^2
\int_t^{\infty}\!\!\!\int_{\rre^2}
\pt_{\eta_j}\phi(\xi,\eta)e^{-i\tau\phi(\xi,\eta)}\frac{1}{p(\xi,\eta)}
B_j(\xi,\eta)\widehat{f}(\xi-\eta)
\widehat{g}(\eta)
d\eta d\tau\nonumber\\
&=&
i\int_t^{\infty}\!\!\!\int_{\rre^2}
\pt_{\tau}\{e^{-i\tau\phi(\xi,\eta)}\}\frac{1}{p(\xi,\eta)}
A(\xi,\eta)\widehat{f}(\xi-\eta)
\widehat{g}(\eta)
d\eta d\tau\nonumber\\
& &+i\sum_{j=1}^2
\int_t^{\infty}\!\!\!\int_{\rre^2}\tau^{-1}
\pt_{\eta_j}\{e^{-i\tau\phi(\xi,\eta)}\}\frac{1}{p(\xi,\eta)}
B_j(\xi,\eta)\widehat{f}(\xi-\eta)
\widehat{g}(\eta)
d\eta d\tau\nonumber\\
&=:&I_{\text{tr}}(f,g)+I_{\text{sr}}(f,g).
\nonumber
\end{eqnarray}

We first evaluate the time resonant term $I_{\text{tr}}(f,g)$. To this end, 
we split it into the following several pieces.
\begin{eqnarray*}
I_{\text{tr}}(f,g)
&=&\sum_{k=1}^3i\int_t^{\infty}\!\!\!\int_{\rre^2}
\pt_{\tau}\{e^{-i\tau\phi(\xi,\eta)}\}\frac{1}{p(\xi,\eta)}
A_k(\xi,\eta)\widehat{f}(\xi-\eta)
\widehat{g}(\eta)
d\eta d\tau\\
&=:&\sum_{k=1}^3I_{\text{tr},k}(f,g),
\end{eqnarray*}
where $A_k$ are defined by (\ref{termA}). 
We treat the term $I_{\text{tr},2}(f,g)$ only 
since the other terms can be treated in a similar way.

Integrating in $\tau$ for $I_{\text{tr},2}(f,g)$, 
we have
\begin{eqnarray*}
\lefteqn{I_{\text{tr},2}(f,g)}\\
&=&
\frac{1}{3}i\lim_{T\to\infty}
\int_t^T\!\!\!\int_{\rre^2}
\pt_{\tau}\{e^{-i\tau\phi(\xi,\eta)}\}\frac{(\xi_1-2\eta_1)^2}{p(\xi,\eta)}
\widehat{f}(\xi-\eta)
\frac{1}{\eta_2^2}\widehat{g}(\eta)
d\eta d\tau\\
&=&
\frac{1}{3}i\lim_{T\to\infty}
\int_{\rre^2}
e^{-iT\phi(\xi,\eta)}\frac{(\xi_1-2\eta_1)^2}{p(\xi,\eta)}
\widehat{f}(\xi-\eta)
\frac{1}{\eta_2^2}\widehat{g}(\eta)
d\eta\\
& &
-\frac{1}{3}i\int_{\rre^2}
e^{-it\phi(\xi,\eta)}\frac{(\xi_1-2\eta_1)^2}{p(\xi,\eta)}
\widehat{f}(\xi-\eta)
\frac{1}{\eta_2^2}\widehat{g}(\eta)
d\eta.
\end{eqnarray*}
Let
\begin{eqnarray*}
m_{\text{tr},2}(\xi,\eta)&:=&\frac{(\xi_1-2\eta_1)^2}{p(\xi,\eta)}\\
&=&\frac{(\xi_1-2\eta_1)^2}{
(\xi_1^2-\xi_1\xi_2+\xi_2^2)
+(\xi_1-2\eta_1)^2+(\xi_2-2\eta_2)^2}
\end{eqnarray*}
and let
\begin{eqnarray*}
\widetilde{m}_{\text{tr},2}(\zeta,\sigma)
&:=&m_{\text{tr},2}(\zeta+\sigma,\sigma)\\
&=&\frac{(\zeta_1-\sigma_1)^2
}{(\zeta_1+\sigma_1)^2-(\zeta_1+\sigma_1)(\zeta_2+\sigma_2)
+(\zeta_2+\sigma_2)^2
+(\zeta_1-\sigma_1)^2
+(\zeta_2-\sigma_2)^2},
\end{eqnarray*}
Then $\widetilde{m}_{\text{tr},2}$ 
satisfies the H\"{o}rmander-Mihlin condition (\ref{MH}). 
Indeed, 
since 
\begin{eqnarray*}
(\zeta_1+\sigma_1)^2-(\zeta_1+\sigma_1)(\zeta_2+\sigma_2)
+(\zeta_2+\sigma_2)^2
+(\zeta_1-\sigma_1)^2
+(\zeta_2-\sigma_2)^2\ge|\zeta|^2+|\sigma|^2,
\end{eqnarray*}
we see that for any $\beta\in{{\mathbb Z}}_+^4$, 
\begin{eqnarray*}
\left|\pt_{\zeta,\sigma}^{\beta}
\widetilde{m}_{\text{tr},2}(\zeta,\sigma)\right|
\lesssim(|\zeta|^2+|\sigma|^2)^{-\frac{|\beta|}{2}}.
\end{eqnarray*}
Hence, Lemma \ref{CM} yields 
\begin{eqnarray}
\lefteqn{\left\|
I_{\text{tr},2}(f,g)
\right\|_{L_{\xi}^2}}
\nonumber\\
&\lesssim&
\lim_{T\to\infty}
\left\|
\int_{\rre^2}
m_{\text{tr},2}(\xi,\eta)
\widehat{[V(T)f]}(\xi-\eta)
\widehat{[V(T)\pt_{x_2}^{-2}g]}(\eta)
d\eta
\right\|_{L_{\xi}^2}
\nonumber\\
& &
+\left\|
\int_{\rre^2}
m_{\text{tr},2}(\xi,\eta)
\widehat{[V(t)f]}(\xi-\eta)
\widehat{[V(t)\pt_{x_2}^{-2}g]}(\eta)
d\eta
\right\|_{L_{\xi}^2}
\nonumber\\
&\lesssim&
\lim_{T\to\infty}
\|V(T)f\|_{L_{x}^{\infty}}
\|V(T)\pt_{x_2}^{-2}g\|_{L_{x}^2}
\nonumber\\
& &
+\|V(t)f\|_{L_{x}^{\infty}}
\|V(t)\pt_{x_2}^{-2}g\|_{L_{x}^2}.
\nonumber
\end{eqnarray}
By the decay estimate (Lemma \ref{lemL} (\ref{linear2})), 
we have
\begin{eqnarray*}
\left\|I_{\text{tr},2}(f,g)
\right\|_{L_{\xi}^2}
\lesssim
t^{-1}\|f\|_{Z}\|g\|_{Z}.
\nonumber
\end{eqnarray*}
In a similar way we have
\begin{eqnarray}
\left\|I_{\text{tr}}(f,g)\right\|_{L_{\xi}^2}
\lesssim t^{-1}\|f\|_{Z}\|g\|_{Z}.\label{LL}
\end{eqnarray}
Next we evaluate the space resonant term 
$I_{\text{sr}}(f,g)$ where $I_{\text{sr}}(f,g)$ is given by (\ref{ST}). 
To this end, we split it into the following pieces.
\begin{eqnarray*}
I_{\text{sr}}(f,g)
&=&\sum_{j=1}^2\sum_{k=1}^3i\int_t^{\infty}\!\!\!\int_{\rre^2}
\tau^{-1}\pt_{\eta_j}\{e^{-i\tau\phi(\xi,\eta)}\}\frac{1}{p(\xi,\eta)}
B_{j,k}(\xi,\eta)\widehat{f}(\xi-\eta)
\widehat{g}(\eta)
d\eta d\tau\\
&=:&\sum_{j=1}^2\sum_{k=1}^3I_{\text{sr},j,k}(f,g),
\end{eqnarray*}
where $B_{j,k}$ are given in (\ref{termB1}) and (\ref{termB2}).  
We treat the term $I_{\text{sr},1,1}(f,g)$ only 
since the other terms can be treated in a similar way.

By integration by parts in $\eta_1$ 
for $I_{\text{sr},1,1}(f,g)$, 
we have
\begin{eqnarray*}
\lefteqn{I_{\text{sr},1,1}(f,g)}\\
&=&
\frac13i
\int_t^{\infty}\!\!\!\int_{\rre^2}\tau^{-1}
\pt_{\eta_1}\{e^{-i\tau\phi(\xi,\eta)}\}
\frac{(\xi_1-2\eta_1)(\sqrt{2}\eta_2-\eta_1)
}{p(\xi,\eta)}\widehat{f}(\xi-\eta)
\frac{(\sqrt{2}\eta_2+\eta_1)}{\eta_2^2}\widehat{g}(\eta)
d\eta d\tau\\
&=&
-\frac13i
\int_t^{\infty}\!\!\!\int_{\rre^2}\tau^{-1}
e^{-i\tau\phi(\xi,\eta)}\pt_{\eta_1}
\left\{\frac{(\xi_1-2\eta_1)(\sqrt{2}\eta_2-\eta_1)}{p(\xi,\eta)}\right\}
\widehat{f}(\xi-\eta)
\frac{(\sqrt{2}\eta_2+\eta_1)}{\eta_2^2}\widehat{g}(\eta)
d\eta d\tau\\
& &
+\frac13i
\int_t^{\infty}\!\!\!\int_{\rre^2}\tau^{-1}
e^{-i\tau\phi(\xi,\eta)}
\frac{(\xi_1-2\eta_1)(\sqrt{2}\eta_2-\eta_1)}{p(\xi,\eta)}
\pt_{\eta_1}\widehat{f}(\xi-\eta)
\frac{(\sqrt{2}\eta_2+\eta_1)}{\eta_2^2}\widehat{g}(\eta)
d\eta d\tau\\
& &
-\frac13i
\int_t^{\infty}\!\!\!\int_{\rre^2}\tau^{-1}
e^{-i\tau\phi(\xi,\eta)}
\frac{(\xi_1-2\eta_1)(\sqrt{2}\eta_2-\eta_1)}{p(\xi,\eta)}
\widehat{f}(\xi-\eta)
\pt_{\eta_1}\left\{\frac{(\sqrt{2}\eta_2+\eta_1)}{\eta_2^2}
\widehat{g}(\eta)\right\}
d\eta d\tau.
\end{eqnarray*}
Let 
\begin{eqnarray*}
m_{\text{sr},1,1}(\xi,\eta)&:=&\frac{(\xi_1-2\eta_1)(\sqrt{2}\eta_2-\eta_1)}{p(\xi,\eta)}\\
&=&\frac{(\xi_1-2\eta_1)(\sqrt{2}\eta_2-\eta_1)
}{(\xi_1^2-\xi_1\xi_2+\xi_2^2)
+(\xi_1-2\eta_1)^2+(\xi_2-2\eta_2)^2}.
\end{eqnarray*}
Then, as in $\widetilde{m}_{\text{tr},2}$, 
we see that $\widetilde{m}_{\text{sr},1,1}(\zeta,\sigma)
=m_{\text{sr},1,1}(\zeta+\sigma,\sigma)$ 
and $\pt_{\sigma_1}\widetilde{m}_{\text{sr},1,1}$ 
satisfy the H\"{o}rmander-Mihlin condition (\ref{MH}). 
Hence Lemma \ref{CM} yields 
\begin{eqnarray}
\lefteqn{\left\|
I_{\text{sr},1,1}(f,g)
\right\|_{L_{\xi}^2}}
\nonumber\\
&\lesssim&
\int_t^{\infty}
\tau^{-1}\left\|
\int_{\rre^2}\pt_{\eta_1}m_{\text{sr},1,1}(\xi,\eta)
\widehat{[V(t)f]}(\xi-\eta)
\widehat{[V(t)(\sqrt{2}\pt_{x_2}^{-1}+\pt_{x_1}\pt_{x_2}^{-2})g]}(\eta)
d\eta 
\right\|_{L_{\xi}^2}d\tau
\nonumber\\
& &+\int_t^{\infty}
\tau^{-1}\left\|
\int_{\rre^2}m_{\text{sr},1,1}(\xi,\eta)
\widehat{[V(t)x_1f]}(\xi-\eta)
\widehat{[V(t)(\sqrt{2}\pt_{x_2}^{-1}+\pt_{x_1}\pt_{x_2}^{-2})g]}(\eta)
d\eta 
\right\|_{L_{\xi}^2}d\tau
\nonumber\\
& &+\int_t^{\infty}
\tau^{-1}\left\|
\int_{\rre^2}m_{\text{sr},1,1}(\xi,\eta)
\widehat{[V(t)f]}(\xi-\eta)
\widehat{[V(t)x_1(\sqrt{2}\pt_{x_2}^{-1}+\pt_{x_1}\pt_{x_2}^{-2})g}](\eta)
d\eta 
\right\|_{L_{\xi}^2}d\tau
\nonumber\\
&\lesssim&
\int_t^{\infty}
\tau^{-1}
\|V(\tau)f\|_{L_{x}^{\infty}}
(\|V(\tau)\pt_{x_2}^{-1}g\|_{L_{x}^2}
+\|V(\tau)\pt_{x_1}\pt_{x_2}^{-2}g\|_{L_{x}^2})d\tau
\nonumber\\
& &
+
\int_t^{\infty}
\tau^{-1}
\|V(\tau)x_1f\|_{L_{x}^{\infty}}
(\|V(\tau)\pt_{x_2}^{-1}g\|_{L_{x}^2}
+\|V(\tau)\pt_{x_1}\pt_{x_2}^{-2}g\|_{L_{x}^2})d\tau
\nonumber\\
& &
+\int_t^{\infty}
\tau^{-1}
\|V(\tau)f\|_{L_{x}^{\infty}}
(\|V(\tau)x_1\pt_{x_2}^{-1}g\|_{L_{x}^2}
+\|V(\tau)x_1\pt_{x_1}\pt_{x_2}^{-2}g\|_{L_{x}^2})d\tau.
\nonumber
\end{eqnarray}
By the decay estimate (Lemma \ref{lemL} (\ref{linear2})), we have
\begin{eqnarray*}
\left\|I_{\text{sr},1,1}(f,g)
\right\|_{L_{\xi}^2}
&\lesssim&
\|f\|_{Z}\|g\|_{Z}\int_t^{\infty}\tau^{-2}d\tau\label{L}\\
&\lesssim&t^{-1}\|f\|_{Z}\|g\|_{Z}.
\nonumber
\end{eqnarray*}
In a similar way, we have
\begin{eqnarray}
\left\|I_{\text{sr}}(f,g)\right\|_{L_{\xi}^2}
\lesssim t^{-1}\|f\|_{Z}\|g\|_{Z}.\label{L}
\end{eqnarray}
Combining (\ref{o1}), (\ref{ST}), (\ref{LL}) and (\ref{L}), 
we have 
\begin{eqnarray*}
\lefteqn{\left\|(\pt_{x_1}+\pt_{x_2})\int_t^{\infty}V(t-\tau)
\left[(V(\tau)f)(V(\tau)g)\right]d\tau\right\|_{L_x^2}}
\qquad\qquad\qquad\\
&=&
\left\|I(f,g)
\right\|_{L_x^2}\\
&\lesssim&
\left\|I_{\text{tr}}(f,g)\right\|_{L_{\xi}^2}
+\left\|I_{\text{sr}}(f,g)\right\|_{L_{\xi}^2}\\
&\lesssim& t^{-1}\|f\|_{Z}\|g\|_{Z}.
\end{eqnarray*}
Hence we obtain (\ref{xFe}). 
This completes the proof of Lemma \ref{xF}.
\end{proof}

\vskip2mm

\begin{lem}\label{cor:v2}
Let $v_1$ and $v_2$ be given by (\ref{v1}) and (\ref{v2}). 
Then for any $t>0$, we have
\begin{eqnarray}
\|v_1(t)\|_{W_x^{3,\infty}}&\lesssim&t^{-1}\|v_{\infty}\|_{X},
\label{v11}\\
\|v_2(t)\|_{H_x^3}&\lesssim&t^{-1}\|v_{\infty}\|_{X}^2,
\label{v22}
\end{eqnarray}
where the norm $\|\cdot\|_{X}$ is given by (\ref{normX}). 
\end{lem}

\vskip2mm

\begin{proof}[Proof of Lemma \ref{cor:v2}] 
By Lemma \ref{lemL} (\ref{linear2}), we have
\begin{eqnarray*}
\|v_1(t)\|_{W_x^{3,\infty}}
=\|V(t)v_{\infty}\|_{W_x^{3,\infty}}
\lesssim t^{-1}
\|\pt_{x_1}^{-\frac12}\pt_{x_2}^{-\frac12}v_{\infty}\|_{W_x^{3,1}},
\end{eqnarray*}
which yields (\ref{v11}).

To show (\ref{v22}), we note
\begin{eqnarray*}
\|v_2(t)\|_{H_x^3}
\lesssim\|v_2(t)\|_{L_x^2}
+\|\nabla v_2(t)\|_{L_x^2}
+\|\Delta v_2(t)\|_{L_x^2}
+\|\nabla\Delta v_2(t)\|_{L_x^2}.
\end{eqnarray*}
We define
\begin{eqnarray*}
B(f,g):=(\pt_{x_1}+\pt_{x_2})\int_t^{\infty}V(t-\tau)
\left[(V(\tau)f)(V(\tau)g)\right]d\tau.
\end{eqnarray*}
Since $\xi=(\xi-\eta)+\eta$, we see
\begin{eqnarray*} 
\nabla B(f,g)&=&
{{\mathcal F}}_{\xi\mapsto x}^{-1}
\left[(\xi_1+\xi_2)\xi e^{it(\xi_1^3+\xi_2^3)}
\int_t^{\infty}\!\!\!\int_{\rre^2}e^{-i\tau\phi(\xi,\eta)}
\widehat{f}(\xi-\eta)\widehat{g}(\eta)
d\eta d\tau\right](x)\\
&=&
{{\mathcal F}}_{\xi\mapsto x}^{-1}
\left[(\xi_1+\xi_2)e^{it(\xi_1^3+\xi_2^3)}
\int_t^{\infty}\!\!\!\int_{\rre^2}e^{-i\tau\phi(\xi,\eta)}
(\xi-\eta)\widehat{f}(\xi-\eta)\widehat{g}(\eta)
d\eta d\tau\right](x)\\
& &
+{{\mathcal F}}_{\xi\mapsto x}^{-1}
\left[(\xi_1+\xi_2)e^{it(\xi_1^3+\xi_2^3)}
\int_t^{\infty}\!\!\!\int_{\rre^2}e^{-i\tau\phi(\xi,\eta)}
\widehat{f}(\xi-\eta)\eta\widehat{g}(\eta)
d\eta d\tau\right](x)\\
&=&B(\nabla f,g)+B(f,\nabla g).
\end{eqnarray*} 
In a similar way, using the identity 
$|\xi|^2=|\xi-\eta|^2+2(\xi-\eta)\cdot\eta+|\eta|^2$, 
we have
\begin{eqnarray*}
\Delta B(f,g)&=&B(\Delta f,g)+2B(\nabla f,\nabla g)+B(f,\Delta g),\\
\nabla\Delta B(f,g)&=&B(\nabla\Delta f,g)+B(\Delta f,\nabla g)+
2\sum_{j=1}^2B(\nabla \pt_{x_j}f,\pt_{x_j} g)\\
& &
+2\sum_{j=1}^2B(\pt_{x_j}f,\nabla \pt_{x_j} g)
+B(\nabla f,\Delta g)+B(f,\nabla\Delta g).
\end{eqnarray*}
Combining Lemma \ref{xF} (\ref{xFe}) with the above 
identities, we have (\ref{v22}).
\end{proof}

%
%
\section{Proof of Theorem \ref{main}}

In this section we complete the proof of Theorem \ref{main}.

\begin{proof}[Proof of Theorem \ref{main}.] 
To prove Theorem \ref{main}, we show the existence 
of solution $w$ to (\ref{ZK11}) with $w\to0$ in $H^2(\rre^2)$ 
as $t\to\infty$. 

Let 
\begin{eqnarray*}
N(w,v_1,v_2)
&:=&
(\pt_{x_1}+\pt_{x_2})(w^2)
+2
(\pt_{x_1}+\pt_{x_2})
\left\{(v_1+v_2)w\right\}\\
& &+
(\pt_{x_1}+\pt_{x_2})(2v_1v_2+v_2^2).
\end{eqnarray*}
Then (\ref{ZK11}) can be rewritten as 
\begin{eqnarray} 
\partial_tw+(\pt_{x_1}^3+\pt_{x_2}^3)w
=N(w,v_1,v_2).
\label{ZK12}
\end{eqnarray}
To show the existence of solution $w$ to (\ref{ZK11}) 
with $w\to0$ in $H^2(\rre^2)$, we consider 
the regularized equation associated with (\ref{ZK12}) : 
\begin{eqnarray}
\partial_tw_{\nu,\lambda}
+(\pt_{x_1}^3+\pt_{x_2}^3)w_{\nu,\lambda}
=(1+\lambda t)^{-5}\rho_{\nu}\ast 
N(\rho_{\nu}\ast w,\rho_{\nu}\ast v_1,\rho_{\nu}\ast v_2),
\label{regular}
\end{eqnarray}
where $0<\lambda,\nu<1$, 
$\rho\in C_0^{\infty}(\rre^2)$ satisfies $\rho\ge0$ 
and $\int\rho(x)dx=1$ and $\rho_{\nu}(x)=
\nu^{-2}\rho(x/\nu)$. 

Thanks to the regularizing factor $\rho_{\nu}\ast$ 
and the time decaying factor $(1+\lambda t)^{-5}$, by using the contraction 
mapping principle, we easily see that for any 
$0<\nu<1$ and $0<\lambda<1$, there exists a 
$T_{\nu,\lambda}>0$ such that (\ref{regular}) has a 
unique solution $w_{\nu,\lambda}$ satisfying 
\begin{eqnarray*}
& &w_{\nu,\lambda}\in\bigcap_{j=1}^{\infty}
C^1([T_{\nu,\lambda},\infty),H_{x}^j),\\
& &\sup_{t\ge T_{\nu,\lambda}}(1+\lambda t)^4\sum_{3i+j\le 3}
\|\pt_t^i\nabla_{x}^jw_{\nu,\lambda}(t)\|_{L_{x}^2}<+\infty.
\end{eqnarray*}
Again using the regularizing and time decaying factors, the above 
solution $w_{\nu,\lambda}$ can be extend to $[0,\infty)$ 
without the smallness assumption on $v_{\infty}$.

We next derive an a priori estimates for $w_{\nu,\lambda}$ 
independent of $\nu$ and $\lambda$ under the assumption 
that $\|v_{\infty}\|_X\le\varepsilon$ with $\varepsilon>0$ 
sufficiently small. 
We abbreviate $w_{\nu,\lambda}$ to $w$. Let 
\begin{eqnarray*}
\|w\|_{X_T}
:=\sup_{t\in[T,\infty)}t^{\alpha}
\left(\|w(t)\|_{H_{x}^2}
+\|w(\tau)\|_{L_{\tau}^3(t,\infty;W_{x}^{1,\infty})}
\right),
\end{eqnarray*}
where $\alpha>0$ and $T>0$. 
We first derive the estimates for $w$ in $H_{x}^2$. 
Classical energy method implies
\begin{eqnarray*}
\|w(t)\|_{H_{x}^2}^2
&\lesssim&\int_t^{+\infty}
(\|w(\tau)\|_{W_{x}^{1,\infty}}
+\|v_1(\tau)\|_{W_{x}^{3,\infty}}
+\|v_2(\tau)\|_{H_x^3})
\|w(\tau)\|_{H_{x}^2}^2d\tau\\
& &+\int_t^{+\infty}
(\|v_1(\tau)\|_{W_{x}^{3,\infty}}\|v_2(\tau)\|_{H_x^3}
+\|v_2(\tau)\|_{H_x^3}^2)
\|w(\tau)\|_{H_{x}^2}d\tau.
\end{eqnarray*} 
By the H\"older inequality and Lemma \ref{cor:v2}, we obtain 
\begin{eqnarray}
\|w(t)\|_{H_{x}^2}^2&\lesssim&
\|w(\tau)\|_{L_{\tau}^3(t,\infty;W_{x}^{1,\infty})}
\|w(\tau)\|_{L_{\tau}^{3}(t,\infty;H_{x}^2)}^2
\nonumber\\
& &+\varepsilon\int_t^{+\infty}\tau^{-1}
\|w(\tau)\|_{H_{x}^2}^2d\tau
+\varepsilon^3\int_t^{+\infty}\tau^{-2}
\|w(\tau)\|_{H_{x}^2}d\tau
\nonumber\\
&\lesssim&t^{-3\alpha+\frac23}
\left(\sup_{t\in[T,\infty)}t^{\alpha}\|w(t)\|_{H_{x}^2}\right)^2
\left(\sup_{t\in[T,\infty)}t^{\alpha}
\|w(\tau)\|_{L_{\tau}^3(t,\infty;W_{x}^{1,\infty})}\right)
\nonumber\\
& &
+\varepsilon t^{-2\alpha}
\left(\sup_{t\in[T,\infty)}t^{\alpha}\|w(t)\|_{H_{x}^2}\right)^2
+\varepsilon^3t^{-\alpha-1}
\left(\sup_{t\in[T,\infty)}t^{\alpha}\|w(t)\|_{H_{x}^2}\right)
\nonumber\\
&\lesssim&t^{-3\alpha+\frac23}\|w\|_{X_T}^3
+\varepsilon t^{-2\alpha}\|w\|_{X_T}^2
+\varepsilon^3t^{-\alpha-1}\|w\|_{X_T},
\nonumber
\end{eqnarray}
where the implicit constants are independent of $\lambda$ and $\nu$. 
Hence
\begin{eqnarray}
\sup_{t\in[T,\infty)}t^{\alpha}\|w(t)\|_{H_{x}^2}
\lesssim T^{-\alpha+\frac23}\|w\|_{X_T}^2
+\varepsilon \|w\|_{X_T}+\varepsilon^3 T^{\alpha-1}.
\label{5.12}
\end{eqnarray}
Next we derive the estimates for $w$ 
in $L^3(t,\infty;W_{x}^{1,\infty})$. 
Since $w$ satisfies 
\begin{eqnarray*}
w(t)&=&-
(\pt_{x_1}+\pt_{x_2})
\int_t^{\infty}V(t-\tau)\left[w(\tau)^2\right]d\tau\\
& &-2
(\pt_{x_1}+\pt_{x_2})\int_t^{\infty}V(t-\tau)
\left[(v_1(\tau)+v_2(\tau))w(\tau)\right]d\tau\\
& &-
(\pt_{x_1}+\pt_{x_2})\int_t^{\infty}V(t-\tau)
\left[2v_1(\tau)v_2(\tau)+v_2(\tau)^2\right]d\tau, 
\end{eqnarray*}
applying the Strichartz estimates (Lemma \ref{lemL} (ii)), 
we have
\begin{eqnarray}
\|w(\tau)\|_{L_{\tau}^3(t,\infty;W_{x}^{1,\infty})}
&\lesssim&\|w(\tau)^2\|_{L_{\tau}^{\frac32}(t,\infty;H_{x}^{2,1})}
+\left\|(v_1(\tau)+v_2(\tau))w(\tau)\right\|_{L_{\tau}^1(t,\infty;H_{x}^2)}
\label{cF}\\
& &
+\|2v_1(\tau)v_2(\tau)+v_2(\tau)^2\|_{L_{\tau}^1(t,\infty;H_{x}^{2})}.
\nonumber
\end{eqnarray}
By the H\"{o}lder inequality, 
\begin{eqnarray}
\|w(\tau)^2\|_{L_{\tau}^{\frac32}(t,\infty;W_{x}^{2,1})}
&\lesssim&\|\|w(\tau)\|_{H_{x}^2}^2\|_{L_{\tau}^{\frac32}(t,\infty)}
\label{t1}\\
&\lesssim&
\left(\sup_{t\in[T,\infty)}t^{\alpha}\|w(t)\|_{H_{x}^2}\right)^2
\|\tau^{-2\alpha}\|_{L_{\tau}^{\frac32}(t,\infty)}
\nonumber\\
&\lesssim&t^{-2\alpha+\frac{2}{3}}
\left(\sup_{t\in[T,\infty)}t^{\alpha}\|w(t)\|_{H_{x}^2}\right)^2
\nonumber\\
&\lesssim&t^{-2\alpha+\frac{2}{3}}\|w\|_{X_T}^2,
\nonumber
\end{eqnarray}
where $\alpha>2/3$. 
By the H\"{o}lder inequality and Lemma \ref{cor:v2}, 
\begin{eqnarray}
\lefteqn{\|(v_1(\tau)+v_2(\tau))w(\tau)\|_{L^1(t,\infty;H_{x}^2)}}
\qquad\qquad
\label{t2}\\
&\lesssim&
\left\|(\|v_1(\tau)\|_{H_{x}^{2,\infty}}
+\|v_2(\tau)\|_{H_x^2})
\|w(\tau)\|_{H_{x}^2}\right\|_{L_{\tau}^1(t,\infty)}
\nonumber\\
&\lesssim&\varepsilon
\left(\sup_{t\in[T,\infty)}t^{\alpha}\|w(t)\|_{H_{x}^2}\right)
\|\tau^{-1-\alpha}\|_{L_{\tau}^1(t,\infty)}
\nonumber\\
&\lesssim&\varepsilon t^{-\alpha}
\left(\sup_{t\in[T,\infty)}t^{\alpha}\|w(t)\|_{H_{x}^2}\right)
\nonumber\\
&\lesssim&\varepsilon t^{-\alpha}
\|w\|_{X_T},\nonumber\\
\lefteqn{\|2v_1(\tau)v_2(\tau)+v_2(\tau)^2\|_{L^1(t,\infty;H_{x}^2)}}
\qquad\qquad
\label{t3}\\
&\lesssim&
\left\|\|v_1(\tau)\|_{H_{x}^{2,\infty}}
\|v_2(\tau)\|_{H_x^2}
+\|v_2(\tau)\|_{H_x^2}^2\right\|_{L_{\tau}^1(t,\infty)}
\nonumber\\
&\lesssim&\varepsilon^3\|\tau^{-2}\|_{L_{\tau}^1(t,\infty)}
\nonumber\\
&\lesssim&\varepsilon^3t^{-1}.
\nonumber
\end{eqnarray} 
Substituting (\ref{t1}), (\ref{t2}), (\ref{t3}) into (\ref{cF}), we have
\begin{eqnarray}
\sup_{t\in[T,\infty)}t^{\alpha}\|w(\tau)\|_{L_{\tau}^3(t,\infty;W_{x}^{1,\infty})}
\lesssim
T^{-\alpha+\frac{2}{3}}\|w\|_{X_T}^2
+\varepsilon\|w\|_{X_T} +\varepsilon^3T^{\alpha-1}.
\label{5.4}
\end{eqnarray}
where the implicit constants are independent of 
$\nu$ and $\lambda$.

By (\ref{5.12}) and (\ref{5.4}), we have 
\begin{eqnarray*}
\|w\|_{X_T}\lesssim 
T^{-\alpha+\frac{2}{3}}\|w\|_{X_T}^2
+\varepsilon\|w\|_{X_T}+\varepsilon^2T^{\alpha-1}.
\end{eqnarray*}
Hence, there exist $\varepsilon>0$ and $T>0$ 
which are independent of $\lambda$ and $\nu$ such that 
for any $0<\lambda<1$ and $0<\nu<1$,
\begin{eqnarray}
\|w\|_{X_T}=\sup_{t\ge T}t^{\alpha}(\|w(t)\|_{H_{x}^2}+
\|w(\tau)\|_{L_{\tau}^3(t,\infty;W_{x}^{1,\infty})})
\le 2\epsilon,\label{5.5}
\end{eqnarray}
where $2/3<\alpha<1$. 
By combining a priori estimate (\ref{5.5}) with 
the standard compactness argument 
(see \cite[Section 3]{OT} for instance), we find that 
there exists a unique solution 
$v\in C([T,\infty) ; H_x^1(\rre))$ to (\ref{ZK}) 
which satisfies $\|w\|_{X_T}
=\|v-v_1-v_2\|_{X_T}<2\varepsilon$.  
By conservation of the energy (\ref{energy}),  
we see $v\in C(\rre ; H^1(\rre))$. 
Furthermore, from the above inequality 
and Lemma \ref{cor:v2}, we see
\begin{eqnarray*}
\|v(t)-V(t)v_{\infty}\|_{H_x^2}&\lesssim&
\|w(t)\|_{H_x^2}+\|v_2(t)\|_{H_x^2}\\
&\lesssim&
\varepsilon t^{-\alpha}+\varepsilon^2 t^{-1}\\
&\lesssim&\varepsilon t^{-\alpha}.
\end{eqnarray*}
This completes 
the proof of Theorem \ref{main}. \end{proof}


\vskip2mm

\subsection*{Acknowledgements} 
J.S was supported by JSPS KAKENHI Grant Numbers 
JP25H00597 and  JP23K20805.

\end{document}